\documentclass[11pt,a4paper]{article}
\usepackage[OT1]{fontenc}
\usepackage[english]{babel}
\usepackage{amsfonts}
\usepackage{amsthm}
\usepackage{amsmath}
\usepackage{needspace}

\usepackage[hidelinks]{hyperref}
\usepackage{xcolor}

\def\h{ {\cal H} }
\def\z{ {\cal Z} }

\def\b{ {\cal B} }
\def\u{ {\cal U} }

\def\t{ {\cal T} }
\def\s{ {\cal S} }

\def\p{ {\cal P} }

\def\k{ {\cal K} }
\def\f{ {\cal F} }
\def\c{ {\cal C} }
\def\cald{ {\cal D} }
\def\j{ {\cal J} }

\newtheorem{teo}{Theorem}[section]
\newtheorem{prop}[teo]{Proposition}
\newtheorem{lem}[teo]{Lemma}
\newtheorem{coro}[teo]{Corollary}
\newtheorem{defi}[teo]{Definition}
\theoremstyle{definition}
\newtheorem{rem}[teo]{Remark}
\newtheorem{ejem}[teo]{Example}

\newtheorem{Questions}[teo]{Questions}

\title{Graphs of operators as points in the Grassmann manifold}
\author{
	Esteban Andruchow\footnotemark[1]
	\and
	L\'azaro Recht\footnotemark[2]
	\and
	Alejandro Varela\footnotemark[3]
}
\date{}

\begin{document}
	\renewcommand{\thefootnote}{\fnsymbol{footnote}}
	
	\maketitle
	
	\footnotetext[1]{Instituto Argentino de Matem\'atica, `Alberto P. Calder\'on', CONICET, Saavedra 15 3er. piso, (1083) Buenos Aires, 
and Universidad Nacional de General Sarmiento, J.M. Gutierrez 1150, (1613) Los Polvorines, Argentina

e-mail: eandruch@campus.ungs.edu.ar}
	\footnotetext[2]{Instituto Argentino de Matem\'atica, `Alberto P. Calder\'on', CONICET, Saavedra 15 3er. piso, (1083) Buenos Aires, Argentina

e-mail: lrecht@gmail.com}

\footnotetext[3]{Instituto Argentino de Matem\'atica, `Alberto P. Calder\'on', CONICET, Saavedra 15 3er. piso, (1083) Buenos Aires, 
and Universidad Nacional de General Sarmiento, J.M. Gutierrez 1150, (1613) Los Polvorines, Argentina

e-mail: avarela@campus.ungs.edu.ar}

\begin{abstract}
We study the set $\Gamma$ of graphs of closed, densely defined operators in a Hilbert space $\h$, regarded as a subset of the Grassmann manifold $\p(\h\times\h)$ of orthogonal projections in $\h\times\h$. We show that the subset $\Gamma^b$ of graphs of bounded operators is the open unit ball of $\p(\h\times\h)$ centered at the graph of the zero operator $P_0$ (which projects onto $\h\times\{0\}$).  This ball is diffeomorphic to $\b(\h)$ via the map $T\mapsto P_T$ ($=$ the projection onto the graph ${Gr(T)}$ of $T$).  We  show that graphs of unbounded closed operators lie at the  boundary of $\Gamma$. We also study the existence and characteristics of minimal geodesics of $\p(\h\times\h)$ joining two graphs $Gr(A)$, $Gr(B)$. If $A,B$ are selfadjoint, such a geodesic always exists, and we construct explicitly a distinguished exponent using the five-space decomposition associated to the pair of subspaces $Gr(A)$, $Gr(B)$. An explicit low-dimensional example shows that the geodesic joining two graphs need not remain inside $\Gamma$, i.e., does not consist entirely of graphs. We also relate graphs of compact operators to the restricted Grassmannian, and study the problem of common complements for pairs $Gr(S)$, $Gr(T)$, giving positive results when one operator is bounded or under lower boundedness conditions.
\end{abstract}

\bigskip

{\bf 2020 MSC: }  58B20, 47A05, 47B25, 53C22.

{\bf Keywords:}  Projection, closed operator, graph of an operator, Grassmann manifold, geodesics.

\section{Introduction}
Let $T:D(T)\subset\h\to\h$ be a possibly unbounded closed operator acting on the complex Hilbert space $\h$, with dense domain $D(T)$. Denote by $Gr(T)$ the graph of $T$:
$$
Gr(T)=\{(f,Tf)\in\h\times\h: f\in D(T)\},
$$
which is a closed subspace of $\h\times\h$. 

Given a Hilbert space $\j$, we denote by $\p(\j)$ the set of closed subspaces of $\j$. To each subspace $\s\in\p(\j)$, corresponds a unique projection, the orthogonal projection $P_\s$ onto $\s$. We use this one to one correspondence to parametrize and describe $\p(\j)$, thus writing 
$$
\p(\j)=\{P_\s: \s \hbox{ closed subspace of } \j\}=\{P\in\b(\j):  P^2=P^*=P\},
$$
where $\b(\j)$ denotes the algebra of bounded operators acting in $\j$.

In this paper we study the position of the graphs $Gr(T)$ as points of $\p(\h\times\h)$. We abbreviate $P_T:=P_{Gr(T)}$, the orthogonal projection onto $Gr(T)$. This work is  a follow up of Section 5 of the paper \cite{eder}, where this matter was considered.  We will study in particular the subsets of $\p(\h\times\h)$ consisting of graphs of operators that are selfadjoint  or bounded. In symbols
\begin{equation}\label{graficas no acotadas}
\Gamma=\{P_T: T:D(T)\subset \h\to \h \hbox{ is closed }\}, \  \ \ \ \ \Gamma_s=\{P_A\in\Gamma:  A \hbox{ is selfadjoint}\}, 
\end{equation}

\begin{equation}\label{graficas acotadas}
\Gamma^b=\{P_S: S \hbox{ is bounded}\} \ \  \ \hbox{ and } \  \ \ \Gamma^b_s=\{P_B\in\Gamma_s: B \hbox{ is bounded}\}.
\end{equation}

Denote by $\Pi_1=P_{0}$  the projection onto $\h\times\{0\}$ (the graph of the $0$ operator). We will see  (Theorem 3.2) that $\Gamma^b$ is the unit ball  
$$
\Gamma^b=\{P\in\p(\h\times\h): \|P-\Pi_1\|<1\}.
$$
This observation enables the fact that the map
$$
\b(\h)\ni T\to P_T\in\Gamma^b
$$
is a $C^\infty$ diffeomorphism (Proposition 3.3). As a consequence, $\Gamma^b_s$ is a complemented submanifold of $\b(\h\times\h)$. Graphs of unbounded closed operators lie in the frontier $\partial\Gamma$ of $\Gamma$ (Proposition 3.7). 

The contents of the rest of the paper are, briefly, the following. The geodesics of the space of projections of a C$^*$-algebra have been thoroughly studied: \cite{pr}, \cite{cpr proyecciones} for an abstract C$^*$-algebra, or for instance \cite{p-q}, \cite{grassmann}, for the specific case of the algebra $\b(\j)$. In Sections 4 and 5 we examine geodesics between graphs. We consider the following questions: 1) do there exist geodesics between graphs?; 2) if the answer is positive, do any of  these geodesics remain in $\Gamma$? We also   consider graphs of compact operators (Section 6). In Section 7 we consider unitary equivalent operators, and show that their graphs can be joined by a smooth curve of graphs. In Section 8 we characterize graphs which have a common complement. In Section 9 we study an unbounded idempotent corresponding to the graph of a closed unbounded operator.

\section{Preliminaries}

If $T$ is a linear operator with domain $D(T)$, we denote by $R(T)$ and $N(T)$ the range and nullspace of $T$, and by $r(T)=\dim R(T)$ and $n(T)=\dim N(T)$ the  rank and nullity of $T$ . Let $\s,\t$ be subspaces of a Hilbert space $\j$,  $\s+\t$ in general denotes the sum, $\s\dot{+}\t$  the direct sum (when $\s+\t$ is closed and $\s\cap\t=\{0\}$), and $\s\oplus\t$ the orthogonal sum  (when $\s\perp\t$). If $\j=\s\dot{+}\t$, we denote by $P_{\s\parallel\t}$ the projection onto $\s$ with nullspace $\t$.
$\b(\j)$ denotes the algebra of bounded linear operators, $\u(\j)$ the group of unitary operators acting in the Hilbert space $\j$. 

We briefly describe in the following remark the differential geometry of the Grassmann manifold $\p(\j)$, based on the papers \cite{pr} and \cite{cpr proyecciones}. These papers deal in fact with the set of projections of an arbitrary unital C$^*$-algebra,   see  \cite{grassmann} for an  abridged description focused on the algebra $\b(\j)$.

\needspace{5\baselineskip}
\begin{rem}\label{21}
	\leavevmode
	\begin{enumerate}
\item
The unitary group $\u(\j)$ of $\j$ acts on $\p(\j)$: $U\cdot P=UPU^*$ (or equivalently, as subspaces: if $P=P_\s$, $U\cdot \s=U\s$). The action is locally transitive: if $\|P-Q\|<1$, then there exists a unitary $U=U(P,Q)$ (which can be chosen as a $C^\infty$ map on $P, Q$) such that $UPU^*=Q$. Since the group $\u(\j)$ is connected, the orbits of this action are the connected components of $\p(\j)$.
\item
The connected components of $\p(\j)$ are parametrized by the rank and co-rank. If $i,j\in\mathbb{N}\cup\{\infty\}$, such that $i+j=\infty$, the connected components
are
$$
\p_{i}^j:=\{P\in\p(\j): r(P)=i, n(P)=j\}.
$$
We abbreviate $\p_\infty:=\p_\infty^\infty$.
\item
Each component $\p_i^j$ is a  complemented $C^\infty$ submanifold of $\b(\j)$. Elements $P,Q$ in different components of $\p(\j)$ satisfy (due to item 1.) that $\|P-Q\|=1$. Thus the whole set $\p(\j)$ is a complemented submanifold of $\b(\h)$.  Fix $P_0\in\p_i^j$. Then the map
$$
\pi_{P_0}:\u(\j)\to \p_i^j, \ \pi_{P_0}(U)=UPU^*
$$
is a $C^\infty$ submersion. 
\item
There is a natural linear connection in $\p(\j)$. It is based on the following graduation of $\b(\h)$ in terms of a fixed projection. Namely, if $P\in\p(\j)$ elements of $\b(\j)$ are written as $2\times 2$ matrices in terms of $\j=R(P)\oplus N(P)$. Then put
$$
\cald_P=\{X\in\b(\j): X \hbox{ is diagonal}\}=\{X\in\b(\j): [X,P]=0\}.
$$
and
$$
\c_P=\{Y\in\b(\j): Y \hbox{ is codiagonal}\},
$$
which also can also be described as $\c_P=\{Y\in\b(\j): Y \hbox{ anticommutes with } 2P-I\}$ (note that elements in $\cald_P$ commute with $2P-I$).  Put ${\bf C}_P$ the natural projection from $\b(\j)$ onto $\c_P$:
$$
{\bf C}_P(X)=PXP^\perp+P^\perp XP.
$$ 
The covariant derivative of the linear connection is
$$
\frac{DX}{dt}={\bf C}_{\gamma(t)}\left(\frac{d}{dt}X(t)\right),
$$
where $X(t)$ is a smooth  tangent field along $\gamma(t)$ (meaning that $X(t)$ is a smooth curve of selfadjoint operators, taking values in $\cald_{\gamma(t)}$ for each $t$).
The geodesics of this connections are
$$
\delta(t)=e^{itX}Pe^{-itX},
$$
where $X=X^*\in\c_P$.
\item
If one measures tangent vectors with the usual norm of $\b(\j)$, and computes lengths of smooth curves $\gamma:I\to\p(\j)$ by means of $\ell(\gamma)=\int_{_I} \|\dot{\gamma}(t)\|d t$, then geodesics curves are minimal for time $\frac{|t|}{\|X\|}\le \pi/2$.
\item 
This fact can be found in \cite{p-q}. There exists a geodesic $\delta$ joining $\delta(0)=P$ and $\delta(1)=Q$ if and only if
$$
\dim R(P)\cap N(Q)=\dim N(P)\cap R(Q).
$$
Such geodesic is unique if and only if these dimensions are zero (i.e., the intersections are trivial). In any case, the geodesic is minimal (for the metric described in 5.). 
\end{enumerate}

\end{rem}

If $A:D(A)\subset\h\to\h$ is selfadjoint,  the spectral theorem (see for instance \cite{readsimon}) states that there exists an isometric isomorphism $F:\h\to L^2(\mu)$ of a  regular Borel finite measure space $(M,\mu)$ and a real valued measurable function $\varphi$ in $M$, which is finite (pp), such that $F(D(A))=\{f \in L^2(\mu): f\varphi\in L^2(\mu)\}$, and $FAF^{-1}f=f\varphi$, for $f\in F(D(A))$. The functional calculus in $A$ can be presented as follows: if $\psi:\mathbb{R}\to\mathbb{R}$ is a Borel function, $\psi(A)$ is determined by $F(D(\psi(A))=\{f\in L^2(\mu): \psi(f)\in L^2(\mu)\}$, and $F\psi(A)F^{-1}f=\psi(f)$. In particular, if $\psi$ is  bounded, $\psi(A)\in\b(\h)$.

If $T:D(T)\subset\h\to\h$ is closed with dense domain, then it has a polar decomposition
$$
T=V|T|
$$
where
$V: \overline{R(T^*)}\to\overline{R(T)}$ is a partial isometry and $|T|=(T^*T)^{1/2}$ is a (possibly unbounded) selfadjoint non negative operator with $D(|T|)=D(T)$. The partial isometry is unique with initial space $\overline{R(|T|)}$. Moreover,  
$$
T^*=|T|V^*,  |T^*|=V|T|V^* \hbox{ and } T^*=V^*|T^*|
$$
is the polar decomposition of $T^*$.

Let us now recall  the explicit matrix form of the projection onto $Gr(T)$, for $T$ closed with dense domain (see for instance Proposition 5.2 in \cite{eder}):
\begin{equation}\label{P Gr(T)}
P_{T}=\left(\begin{array}{cc} (I+|T|^2)^{-1} & (I+|T|^2)^{-1}T^* \\ T(I+|T|^2)^{-1} & T(I+|T|^2)^{-1}T^* \end{array}\right).
\end{equation}
Notice that all entries are densely defined bounded operators (we consider them defined in the whole space $\h$, using the same symbol). In particular, if $A$ is selfadjoint,
\begin{equation}\label{P Gr(A)}
P_{A}=\left(\begin{array}{cc} (I+A^2)^{-1} & A(I+A^2)^{-1} \\ A(I+A^2)^{-1} & A^2(I+A^2)^{-1} \end{array}\right).
\end{equation}

\begin{rem}\label{remark 21}
If $T$ is closed and densely defined, and $T=V|T|$ is its polar decomposition, then it is elementary to check that
$$
P_T=\left(\begin{array}{cc} I & 0 \\ 0 & V \end{array} \right) P_{|T|} \left(\begin{array}{cc} I & 0 \\ 0 & V^* \end{array} \right) \ \hbox{ and } \ \left(\begin{array}{cc} I & 0 \\ 0 & V^* \end{array} \right)P_T\left(\begin{array}{cc} I & 0 \\ 0 & V \end{array} \right)=P_{|T|}
$$
(for the right hand equality use that $T^*V=|T|=V^*T$).
In particular, if we denote ${\bf V}:=\left(\begin{array}{cc} I & 0 \\ 0 & V \end{array} \right)$, then ${\bf V}$ is a partial isometry, with initial space $\h\times N(T)^\perp$ and final space $\h\times \overline{R(T)}$.
\end{rem}

The following operator will be useful:
\begin{equation}\label{omega}
\Omega:\h\times\h\to\h\times\h, \ \ \Omega(f,g)=(-g,f).
\end{equation}
Note that $\Omega^2=-I$ and $\Omega^*=-\Omega$, in particular, $\Omega$ is a unitary operator.
\begin{rem}\label{propiedad omega}
The following  fact is elementary (see  Lemma 5.4 of \cite{eder} for a proof): if $T$ is a closed densely defined operator, then 
$$
Gr(T)^\perp=\{(-T^*f,f): f\in D(T^*)\}.
$$
Therefore, we have
\begin{enumerate}
\item
$$
\Omega(Gr(T))^\perp)=\Omega(Gr(T))^\perp=\{(-f,-T^*f): f\in D(T^*)\}=Gr(T^*).
$$
\item
In particular, for $A$  a selfadjoint operator, we have that
$Gr(A)^\perp=\{(-Af,f)\in\h\times\h: f\in D(A)\}$ and 
$$
\Omega(Gr(A))=Gr(A)^\perp.
$$
\item
For $T$ closed and densely defined,
$$
\Omega P_T \Omega^*=P_{T^*}^\perp.
$$
Indeed, 
$$
\Omega P_T\Omega^*=\Omega P_{Gr(T)}\Omega^*=P_{\Omega(Gr(T))}=P_{Gr(T^*)^\perp}= P_{T^*}^\perp.
$$
\end{enumerate}
\end{rem}

Finally, let us point out a few basic properties of $\Gamma$.
\begin{rem}
Let $T$ be a closed densely defined operator.
\begin{enumerate}
\item
The set $\Gamma$ of graphs  lies in the component $\p_{\infty}$ of $\p(\h\times\h)$. Indeed, it is apparent that $\dim Gr(T)=+\infty, \dim Gr(T^*)=+\infty$, and therefore $Gr(T)^\perp=\Omega(Gr(T^*))$ is also infinite dimensional. 
\item
Denote by $\Pi_1, \Pi_2$ the projections onto, respectively, the first and second  coordinates of $\h\times\h$: $\Pi_1(f,g)=(f,0)$,  $\Pi_2(f,g)=(0,g)$.
A graph $Gr(T)$ corresponds to a bounded operator if and only if  $\Pi_1(Gr(T))=\h\times\{0\}$. Indeed, this  condition means that to every $f\in\h$ corresponds a vector $g\in\h$ such that $(f,g)\in Gr(T)$, i.e., $g=Tf$, then $D(T)=\h$, and thus, by the closed graph theorem, $T$ is bounded. The converse is obvious.
\end{enumerate}

\end{rem}

\section{Structure of $\Gamma$}

In this section we shall investigate the structures of the sets $\Gamma, \Gamma_s, \Gamma^b$ and $\Gamma_s^b$. We shall see that $\Gamma^b$ is an open subset of $\p(\h\times\h)$. More precisely, $\Gamma^b$ is the norm-ball of radius $1$ in $\p(\h\times\h)$, centered at the graph $\h\times\{0\}$ of the zero operator:
$$
\Gamma^b=\{P\in\p(\h\times\h): \|P-\Pi_1\|<1\}.
$$
Thus a complemented $C^\infty$-submanifold of $\b(\h\times\h)$.

We shall also prove  that  $\Gamma_s^b$ is a complemented $C^\infty$-submanifold of $\Gamma^b$ (and then also a complemented submanifold of $\b(\h\times\h)$).

To proceed further, we recall the following result (see for instance Theorem 9.35 in the book by F. Deutsch \cite{deutsch2},  or V. Havin and B. J\"oricke \cite{havinjoricke}:
\begin{teo}\label{havin}  {\rm \cite{deutsch2}} 
 
Let $P, Q$ be orthogonal projections in $\j$. Then the following are equivalent:
\begin{enumerate}
\item
$R(P)\cap R(Q)=\{0\}$  and $R(P)+R(Q)$   is closed.
\item
$\|PQ\|<1$.
\end{enumerate}
\end{teo}

As a consequence, we have the following characterization of graphs of bounded operators. This  result is  known, for instance, one can build  a proof using the results in  \cite{chung}. We include an elementary argument, based on the above theorem and the Krein-Krasnoselski-Milman formula.
\begin{teo}\label{42}
Let $\s$ be a closed subspace of $\h\times\h$. Then 
$$P_\s\in\Gamma^b \ \iff \ 
 \|P_\s-\Pi_1\|<1.
$$
\end{teo}
\begin{proof}
Suppose first that $\s=Gr(S)$ for $S$ a bounded operator. Then clearly 
$$
R(P_S)\cap R(\Pi_2)=\{(f,Sf): f\in\h\}\cap\{0\}\times\h=\{(0,0)\},
$$
and for $(f,g)\in\h\times\h$,  
$$
(f,g)=(f,Sf)+(0,g-Sf)\in R(P_\s)+ R(\Pi_2).
$$
In particular, $R(P_\s)+R(\Pi_2)$ is closed. 
Then, by Theorem \ref{havin}, $\|P_\s \Pi_2\|<1$. 
On the other hand, $R(P_\s^\perp)=Gr(S)^\perp$ and $R(\Pi_1)=\h\times\{0\}$ are in direct sum, equal to $\h\times\h$. Indeed,
first note that $Gr(S)^\perp\cap \h\times\{0\}=\{0\}$: if $(-S^*f,f)=(g,0)$, then clearly $g=0$. Also,  any $(f,g)\in\h\times\h$ can be written
$$
(f,g)=(-S^*g,g)+(f+S^*g,0)\in Gr(S)^\perp + \h\times\{0\}.
$$
Therefore, again by Theorem \ref{havin}, $\|P_\s^\perp\Pi_1\|<1$. 
Then,  by the Krein-Krasnoselsky-Milman formula, 
$$
\|P_\s-\Pi_1\|=\max\{\|P_\s-P_\s\Pi_1\|, \|\Pi_1-\Pi_1P_\s\|\}=\max\{\|P_\s\Pi_2\|, \|\p_\s^\perp\Pi_1\|\}<1.
$$

Conversely, suppose that $\|P_\s-\Pi_1\|<1$. By the above formula, we have $\|P_\s\Pi_2\|<1$ and $\|P_\s^\perp \Pi_1\|<1$. The first inequality, again by Theorem \ref{havin}, implies that $\s\cap \{0\}\times\h=\{0\}$ and $\s + \{0\}\times \h$ is closed. The second inequality, likewise, implies that $\s^\perp \cap\ \h\times \{0\}=\{0\}$. Then $\s + \{0\}\times\h=\left(\s^\perp \cap\ \h\times \{0\}\right)^\perp$ is dense in $\h\times\h$. Therefore, 
$$
\s \dot{+} \{0\}\times\h=R(P_\s)\dot{+}R(\Pi_2)=\h\times\h.
$$
The fact that $R(P_\s)\cap R(\Pi_2)=\{(0,0)\}$ implies that $\s$ is the graph of a linear operator $S$: if $(f,g), (f,h)\in\s$, then $(0,g-h)=(f,g)-(f,h)\in \s\cap R(\Pi_2)$, i.e., $g=h$. The fact that $\s+ \{0\}\times \h=\h\times\h$, implies that  
$$
\Pi_1(\s)=\Pi_1(\s+\{0\}\times\h)=\Pi_1(\h\times\h)=\h\times\{0\},
$$
and therefore the domain of this operator $S$ is $\h$. Since the subspace $\s$ is closed, by the closed graph theorem, the operator $S$ is bounded. 
\end{proof}

Once we know that $\Gamma^b=\{P\in\p(\h\times\h): \|P-\Pi_1\|<1\}$, and in particular  is open in $\p(\h\times\h)$, we can establish the following result:

\begin{prop}\label{prop 13}
The map $\gamma:\b(\h)\to \Gamma^b$, $\gamma(T)=P_{Gr(T)}$ is a $C^\infty$  diffeomorphism.
\end{prop}
\begin{proof}
We have that $\gamma(T)=\left(\begin{array}{cc} (I+T^*T)^{-1} &  (I+T^*T)^{-1}T^* \\ T(I+T^*T)^{-1} & T(I+T^*T)^{-1}T^* \end{array} \right)$, which is clearly a $C^\infty$ map. Also it is apparent that it is one to one and onto. In particular $\Gamma^b$ is contained in the open subset 
$$
\b_{1,1}:=\{{\bf T}\in\b(\h\times\h): {\bf T}_{1,1} \hbox{ is invertible in } \b(\h)\}.
$$
On this open set, we can define $\eta:\b_{1,1}\to \b(\h)$, $\eta({\bf T})={\bf T}_{1,2}{\bf T}_{1,1}^{-1}$. This map $\eta$ is $C^\infty$, and clearly 
$$
\eta\big|_{\Gamma^b}:\Gamma^b \to \b(\h)
$$
is an inverse for $\gamma$: $\eta(\gamma(T))=T(I+T^*T)^{-1}(I+T^*T)=T$ and $\gamma(\eta(P_{Gr(T)}))=\gamma(T)=P_{Gr(T)}$.
\end{proof} 

Using this diffeomorphism, we have that any submanifold of $\b(\h)$ has a counterpart inside $\Gamma^b$. Denote
\begin{equation}\label{gamma compactos}
\Gamma^c:=\{P_{Gr(K)}: K\in\b(\h) \ \hbox{ is compact}\} \ \hbox{ and }\  \Gamma_s^c:=\{P_{Gr(K)}\in\Gamma^c: K^*=K\}.
\end{equation}

\begin{teo}

\noindent

\begin{enumerate}
\item
$\Gamma_s^b$ is a closed complemented $C^\infty$ submanifold of $\Gamma^b$ (and therefore also of $\b(\h\times\h))$.
\item
$\Gamma_s^c\subset\Gamma^c$ are closed non complemented submanifolds of $\Gamma^b$ (and of $\b(\h\times\h)$).
\item
Moreover, as  topological spaces, in the norm topology of $\b(\h\times\h)$, $\Gamma^b$, $\Gamma^b_s$,  $\Gamma^c$  and $\Gamma^c_s$  are contractible. 
\end{enumerate}
\end{teo}
\begin{proof}
Using the diffeomorphism $\gamma:\b(\h)\to\Gamma^b$ of Proposition \ref{prop 13}, $\Gamma_s^b=\gamma(\b_s(\h))$, where $\b_s(\h)$ is a closed complemented real linear subspace of $\b(\h)$. Similarly $\Gamma^c=\gamma(\k(\h))$ and $\Gamma_s^c=\gamma(\k_s(\h))$, where $\k(\h)$ and $\k_s(\h)$ are closed non complemented subspaces of $\b(\h)$. Clearly these spaces are contractible.
\end{proof}

Recall the unitary operator $\Omega(f,g)=(-g,f)$, and consider the map
\begin{equation}\label{mapa alfa}
\alpha:\b(\h\times\h)\to\b(\h\times\h), \ \alpha({\bf B})=I+\Omega {\bf B}\Omega
\end{equation}
Note the following properties of the map $\alpha$:
\begin{rem}\label{propiedades alfa}

\noindent

\begin{enumerate}
\item
The map $\alpha$ is affine: it is the identity operator  $I$ plus the (isometric) linear map ${\bf B}\mapsto \Omega {\bf B}\Omega$. Therefore it is analytic.
\item
$\alpha\circ\alpha=I_{\b(\h\times\h)}$:  using that $\Omega^2=-I$,
$$
\alpha(\alpha({\bf B}))=\alpha(I+\Omega {\bf B} \Omega)=I+\Omega(I+\Omega {\bf B} \Omega)\Omega=I+\Omega^2+\Omega^2 {\bf B}\Omega^2={\bf B}.
$$
In particular, $\alpha$ is a diffeomorphism of $\b(\h\times\h)$.
\item
$\alpha$ maps $\p(\h\times\h)$ onto itself. Indeed, for ${\bf P}\in\p(\h\times\h)$,
$$
\alpha({\bf P})=I+\Omega {\bf P} \Omega=\Omega \Omega^* -\Omega {\bf P}\Omega^*=\Omega(I-{\bf P})\Omega^*\in\p(\h\times\h).
$$
\end{enumerate}
\end{rem}
Combining these properties with $\Omega P_{Gr(T)} \Omega^*=P_{Gr(T^*)}^\perp$ in Remark \ref{propiedad omega}.3, we get

\begin{prop}
$$\alpha(P_{Gr(T)})=P_{Gr(T^*)}.$$
\end{prop}
\begin{proof}
According to Remark \ref{propiedades alfa}.3
$$
\alpha(P_{Gr(T)})=\Omega(I-P_{Gr(T)})\Omega^*=I-\Omega P_{Gr(T)}\Omega^*=I-P_{Gr(T^*)}^\perp=P_{Gr(T^*)}.
$$
\end{proof}

Therefore $\alpha$ defines a symmetry inside $\p(\h\times\h)$, and also a symmetry in $\Gamma$, $\Gamma^b$ and $\Gamma^c$. The fixed points of $\alpha$, when restricted to theses three submanifolds are, respectively,  $\Gamma_s$, $\Gamma^b_s$ and $\Gamma^c_s$.

The following result shows that the graphs of unbounded closed operators  lie at the frontier of $\Gamma$. 
\begin{prop}
Let $T:D(T)\subset\h\to \h$ be an unbounded closed operator.   Then $P_T\in\partial \Gamma$, i.e., there exist projections $E\in\p(\h\times\h)\setminus\Gamma$ such that $E$ are arbitrarily close to $P_T$.
\end{prop}
\begin{proof}
Consider the polar decomposition $T=V|T|$.   By the spectral theorem we may suppose that $\h=L^2(M,\mu)$ and $|T|=M_\varphi$, for a non negative measurable (finite) function $\varphi$. The fact that $T$ is unbounded implies that for any $N>0$, the set $\{x\in M: \varphi(x)\ge N\}$ has positive measure. Fix $\epsilon>0$, and denote
$$
\Delta_0=\{x\in M: \varphi(x)=0\}, \hbox{ and } \Delta_j=\{x\in M: \varphi(x)\in((j-1)\epsilon, j\epsilon]\}, \hbox{ for } j\in \mathbb{N}.
$$
Clearly $M=\bigcup_{j\ge 0} \Delta_j$, and the union is disjoint. Consider the step function
$$
\kappa_\epsilon=\sum_{j=1}^\infty \frac{2j-1}{2}\epsilon \chi_{\Delta_j}
$$
(note that $\frac{2j-1}{2}\epsilon$ is the midpoint of the interval $((j-1)\epsilon, j\epsilon]$).
Apparently $|\varphi(x)-\kappa_\epsilon(x)|\le \epsilon$. The functions $f(t)=\frac{1}{1+t^2}$ and $g(t)=\frac{t}{1+t^2}$ satisfy $|f'(t)|, |g'(t)|\le 1$.  Therefore $|f(t)-f(s)|\le |t-s|$ and $|g(t)-g(s)|\le |t-s|$. Thus
 
$$
\|\frac{1}{1+\varphi^2}-\frac{1}{1+\kappa_\epsilon^2}\|_\infty\le\epsilon \ \hbox{ and } \ \|\frac{\varphi}{1+\varphi^2}-\frac{\kappa_\epsilon}{1+\kappa_\epsilon^2}\|_\infty\le\epsilon.
$$
Consider the operator $A_\epsilon=M_{\kappa_\epsilon}$ acting in $L^2(M,\mu)$. Clearly $A_\epsilon$ is positive and block diagonal: $A_\epsilon=\sum_{k=1}^\infty \frac{2k-1}{2}\epsilon P_k$, where $P_k=M_{\chi_{\Delta_k}}\ne 0$ are pairwise orthogonal.
Then we have that 
\begin{itemize}
\item
$\|(I+|T|^2)^{-1}-(I+A_\epsilon^2)^{-1}\|=\|\frac{1}{1+\varphi^2}-\frac{1}{1+\kappa_\epsilon^2}\|_\infty<\epsilon$.
\item
Similarly, $\|(I+|T|^2)^{-1}|T|-(I+A_\epsilon^2)^{-1}A_\epsilon\|=\|\frac{\varphi}{1+\varphi^2}-\frac{\kappa_\epsilon}{1+\kappa_\epsilon^2}\|_\infty<\epsilon$, and therefore
$$
\|(I+|T|^2)^{-1}|T|V^*-(I+A_\epsilon^2)^{-1}A_\epsilon V^*\|\le \|(I+|T|^2)^{-1}|T|-(I+A_\epsilon^2)^{-1}A_\epsilon\|<\epsilon
$$
and 
$$
\|V(I+|T|^2)^{-1}|T|-V(I+A_\epsilon^2)^{-1}A_\epsilon\|=\|V|T|(I+|T|^2)^{-1}-VA_\epsilon(I+A_\epsilon^2)^{-1}\|
$$
$$
\le \|(I+|T|^2)^{-1}|T|-(I+A_\epsilon^2)^{-1}A_\epsilon\|<\epsilon.
$$
\item
Also, since $|T|(I+|T|^2)^{-1}|T|=I-(I+|T|^2)^{-1}$,
$$
\|V|T|(I+|T|^2)^{-1}|T|V^*-VA_\epsilon(I+A_\epsilon^2)^{-1}A_\epsilon V^*\|=\|V\left(|T|(I+|T|^2)^{-1}|T|-A_\epsilon(I+A_\epsilon^2)^{-1}A_\epsilon\right) V^*\|
$$
$$
\le\||T|(I+|T|^2)^{-1}|T|-A_\epsilon(I+A_\epsilon^2)^{-1}A_\epsilon\|=\|(I+|T|^2)^{-1}-(I+A_\epsilon^2)^{-1}\|<\epsilon.
$$
\end{itemize}
Denote, for brevity, $\lambda_k=\frac{2k-1}{2}\epsilon$, and put
$$
P_\epsilon=\left(\begin{array}{cc} (I+A_\epsilon^2)^{-1} & (I+A_\epsilon^2)^{-1}A_\epsilon V^* \\ VA_\epsilon(I+A_\epsilon^2)^{-1} & VA_\epsilon(I+A_\epsilon^2)^{-1}A_\epsilon V^* \end{array}\right)=\sum_{k=1}^\infty\left(\begin{array}{cc} \frac{1}{1+\lambda_k^2}P_k & \frac{\lambda_k}{1+\lambda_k^2}P_kV^* \\ \frac{\lambda_k}{1+\lambda_k^2}VP_k & \frac{\lambda_k^2}{1+\lambda_k^2}VP_kV^* \end{array}\right).
$$
Write $B_\epsilon:=VA_\epsilon$. This product is the  polar decomposition of $B_\epsilon$. Indeed, 
$$
N(A_\epsilon)=\{f\in L^2(M,\mu): f \hbox{ is supported in } \Delta_0\}=N(|T|).
$$
Therefore $\overline{R(A_\epsilon)}=\overline{R(|T|)}$, and then  $V:N(A_\epsilon)\to\overline{R(A_\epsilon)}$ is the partial isometry in the polar decomposition of $A_\epsilon$ as well.

It follows, by formula (\ref{P Gr(T)}),  that  $P_\epsilon$ is the projection onto the graph of $B_\epsilon$. 
Then, using the elementary estimation for block matrices
\begin{equation}\label{pedorrin}
\|\left(\begin{array}{cc} A & B \\ B^* & C \end{array}\right)\|\le \|\left(\begin{array}{cc} A & 0 \\ 0 & C \end{array}\right)\|+\|\left(\begin{array}{cc} 0 & B \\ B^* & 0 \end{array}\right)\|=\max\{\|A\|,\|C\|\}+\|B\|,
\end{equation}
we have that 
$$
\|P_T-P_\epsilon\|<2\epsilon.
$$
Let us show that $P_{\epsilon}$ can be approximated by projections that do not belong to $\Gamma$. 
Fix an integer $N\ge 1$, and put 
$$
E_N=\sum_{k=1}^N\left(\begin{array}{cc} \frac{1}{1+\lambda_k^2}P_k & \frac{\lambda_k}{1+\lambda_k^2}P_kV^* \\ \frac{\lambda_k}{1+\lambda_k^2}VP_k & \frac{\lambda_k^2}{1+\lambda_k^2}VP_kV^* \end{array}\right)+ \sum_{k=N+1}^\infty\left(\begin{array}{cc}  0 & 0 \\ 0 &  VP_kV^*\end{array}\right).
$$
Both matrices are orthogonal projections, the left hand matrix is the series that gives $P_\epsilon$ truncated at $N$. Moreover, their product is clearly zero. It follows that $E_N$ is an orthogonal projection as well.
Let us prove that  $\|P_\epsilon-E_N\|\to 0$ as $N\to \infty$, estimating the norms  of the entries of this difference, and using the elementary inequality (\ref{pedorrin}). 
The $1,1$  entry 
$\sum_{k=N+1}^\infty \frac{1}{1+\lambda_k^2}P_k$ has norm bounded by 
$$
\|\sum_{k=N+1}^\infty \frac{1}{1+\lambda_k^2}P_k\|=\frac{1}{1+\lambda_{N+1}^2}\to 0,
$$
because  the sequence $\lambda_k\nearrow\infty$.
Also, the $1,2$ entry (and thus also the $2,1$ entry) is bounded by
$$
\|\sum_{k=N+1}^\infty \frac{1}{1+\lambda_k^2}P_kV^*\|\le \|\sum_{k=N+1}^\infty \frac{1}{1+\lambda_k^2}P_k\|=\sup_{k\ge N+1}\frac{\lambda_k}{1+\lambda_k^2}\to 0  \ \hbox{ as } \ N\to \infty.
$$
The $2,2$ entry is
$$
\sum_{k=N+1}^\infty \frac{\lambda_k^2}{1+\lambda_k^2}VP_kV^*-\sum_{k=N+1}^\infty VP_kV^*=\sum_{k=N+1}^\infty \left(\frac{\lambda_k^2}{1+\lambda_k^2}-1\right)VP_kV^*=-\sum_{k=N+1}^\infty \frac{1}{1+\lambda_k^2}VP_kV^*
$$
with norm equal to $\frac{1}{1+\lambda_{N+1}^2}\to 0$.

Moreover,  $E_N\notin \Gamma$. Indeed, pick $f\in R(VP_mV^*)$ for $m\ge N+1$. Then
$$
E_N\left(\begin{array}{c} 0 \\ f \end{array}\right)=\left(\begin{array}{c} 0 \\ f \end{array}\right).
$$
The fact that $T$ is unbounded, means that, as mentioned at the beginning of the proof,  the set $\{x\in M: \varphi(x)\ge N\}$ has positive measure, for arbitrarily large $N$, i.e., there are non trivial $VP_mV^*$  for arbitrarily large $m$.
 
Summarizing, we can approximate $P_T$ with projections $E_N$ which lie outside of $\Gamma$.
\end{proof}

\section{Geodesics between graphs}
The following result was proved in \cite{eder} (Theorem 5.7):
\begin{teo} \label{teo eder}
Let $T:D(T)\subset\h\to \h$ be a densely defined closed operator. Then there exists a unique minimal  geodesic $\delta$ between $\delta(0)=\Pi_1=P_{Gr(0)}$ and $\delta(1)=P_{Gr(T)}$. It is given by
$$
\delta(t)=e^{itZ}\Pi_1 e^{-itZ}, \hbox{ for } Z=\left(\begin{array}{cc} 0 & i \arctan(|T|)V^* \\ -iV\arctan(|T|) & 0 \end{array} \right),
$$
where $T=V|T|$ is the polar decomposition.

Moreover, the intermediate projections $\delta(t)$, $t\in(0,1)$ are the  graphs of bounded operators: $\delta(t)=P_{Gr(B(t))}$, where $B(t)\in\b(\h)$.
\end{teo}

\begin{Questions}\label{cuestiones}
We shall next consider the assertions of  Theorem \ref{teo eder}  for arbitrary  pairs of closed operators $S$ and $T$:
\begin{itemize}
\item
Does there exist a minimal geodesic of $\p(\h\times\h)$ joining $Gr(S)$ and $Gr(T)$?
\item
If the answer is positive, can it be chosen as a curve of graphs (i.e., a curve inside $\Gamma$)?
\end{itemize}
\end{Questions}
Concerning the first question, any pair of graphs $Gr(A)$ and $Gr(B)$ of selfadjoint operators $A, B$ can be joined by a geodesic of $\p(\h\times\h)$:
\begin{teo}\label{teo 42}
Let $A:D(A)\subset\h\to\h$ and $B:D(B)\subset\h\to\h$ be selfadjoint operators. Then there exists a minimal geodesic $\delta$ of $\p(\h\times\h)$ such that $\delta(0)=P_{Gr(A)}$ and $\delta(1)=P_{Gr(B)}$. 
\end{teo}
\begin{proof}
We  need to check that $\dim Gr(A)\cap Gr(B)^\perp=\dim Gr(A)^\perp\cap Gr(B)$. But this is clear from Remark \ref{propiedad omega}.2: $\Omega(Gr(A)\cap Gr(B)^\perp)=Gr(A)^\perp\cap Gr(B)$.
\end{proof}

For bounded operators we have the following criterium  (Proposition 5.5 in \cite{eder}):

\begin{lem}
Suppose that $S$ and $T$ are bounded. There exists a (minimal) geodesic between $P_{Gr(S)}$ and $P_{Gr(T)}$ for $S,T\in\b(\h)$ if and only if $\dim N(I+S^*T)=\dim N(I+T^*S)$.
\end{lem}
\begin{coro}
If both $S$ and $T$ are bounded, and one of them is compact, then there exists a geodesic between $Gr(S)$ and $Gr(T)$.
\end{coro}
\begin{proof}
Since $S^*T$ is compact, by the Fredholm alternative 
$$
\dim N(I+S^*T)=\dim N((I+S^*T)^*)=\dim N(I+T^*S).
$$
\end{proof}
On the negative side,we have he following easy example (\cite{eder}, Example 5.6):
\begin{ejem}
 Denote by ${\bf S}$ the unilateral shift in $\ell^2(\mathbb{N})$. Using the above Lemma it is easy to see that the graphs of $-2{\bf S}$ and $T=I$ cannot be joined by a geodesic. 
\end{ejem}

\section{Geodesics of graphs of selfadjoint operators}

We shall focus now in the second question in \ref{cuestiones}, for selfadjoint operators $A$ and $B$.
First,  we examine more closely the construction of the exponents of the geodesics  joining $Gr(A)$ and $Gr(B)$, among which we will distinguish a special exponent $X_{A,B}$ (depending on $A$, $B$ and $\Omega$).  It is based in the 5 space (orthogonal) decomposition of $\h\times\h$ in the presence of two subspaces (see \cite{dixmier} or \cite{davis}, also \cite{halmos}), in this case $Gr(A)$ and $Gr(B)$:
$$
\h\times\h=Gr(A)\cap Gr(B)\  \oplus\  Gr(A)^\perp\cap Gr(B)^\perp\  \oplus\ Gr(A)\cap Gr(B)^\perp\ \oplus\ Gr(A)^\perp\cap Gr(B)\  \oplus\  \h_g,
$$
where $\h_g$ is  the  (so  called) {\it generic part} of $Gr(A)$ and $Gr(B)$. This decomposition reduces both $P_A$ and $P_B$.
\begin{itemize}
\item
On $Gr(A)\cap Gr(B)\  \oplus\  Gr(A)^\perp\cap Gr(B)^\perp$, $P_A$ and $P_B$ coincide: they are both the identity, or both trivial. Therefore one chooses here  $X_{A,B}=0$.
\item In $\h_g$ there is a unique possible choice, which we denote $X'_{A,B}$: the reductions $P_A\big|_{\h_g}$ and $P_B\big|_{\h_g}$ satisfy the condition given in Remark \ref{21}.6, establishing that they can be joined by a unique geodesic.  
\item In $Gr(A)\cap Gr(B)^\perp\ \oplus\ Gr(A)^\perp\cap Gr(B)$, we must pick an isometric isomorphism $W:Gr(A)^\perp\cap Gr(B)\to Gr(A)\cap Gr(B)^\perp$. The exponent in this part is $i\frac{\pi}{2}(-W^*\oplus W)$.  We choose $W=\Omega\big|_{Gr(A)^\perp\cap Gr(B)}$ (recall that   $\Omega(Gr(A)\cap Gr(B)^\perp)=Gr(A)^\perp\cap Gr(B)$). Then $W^*=-\Omega\big|_{Gr(A)\cap Gr(B)^\perp}$, and thus $i\frac{\pi}{2}(-W^*\oplus W)=i\frac{\pi}{2}\Omega$ in this part.
\end{itemize}

Summarizing, we  choose this distinguished exponent just constructed:
\begin{equation}\label{X_A,B}
X_{A,B}=0  \ \oplus \ -i\Omega  \ \oplus X'_{A,B},
\end{equation}
in the three space decomposition 
$$
\Big(Gr(A)\cap Gr(B)\  \oplus\  Gr(A)^\perp\cap Gr(B)^\perp\Big) \oplus \Big(Gr(A)\cap Gr(B)^\perp\ \oplus\ Gr(A)^\perp\cap Gr(B)\Big)\oplus\h_g.
$$
 \begin{lem}
$X_{A,B}$ commutes with $\Omega$
\end{lem}
\begin{proof}
Again, the fact that $\Omega(Gr(A))=Gr(A)^\perp$, means that $\Omega$ leaves invariant the first two subspaces above, 
$\Big(Gr(A)\cap Gr(B)\  \oplus\  Gr(A)^\perp\cap Gr(B)^\perp\Big)$ and $\Big(Gr(A)\cap Gr(B)^\perp\ \oplus\ Gr(A)^\perp\cap Gr(B)\Big)$. Therefore it leaves invariant also the third $\h_g$. 

Thus it suffices to check that $\Omega$ commutes with $X_{A,B}$ in these three subspaces. This is clear in the first summand, where $X_{A,B}=0$, and in the second, where $X_{A,B}=i\frac{\pi}{2}\Omega$. Let us examine what happens in $\h_g$. We know that $X'_{A,B}$ is the unique selfadjoint operator in $\h_g$, $\|X'_{A,B}\|\le\pi/2$, which is codiagonal with respect to   
the reduction $P'_{A}:=P_{A}\big|_{\h_g}$ of $P_{A}$ to $\h_g$, and satisfies
$$
e^{iX'_{A,B}}P'_{A}e^{-iX'_{A,B}}=P'_{B}.
$$
Then, conjugating with $\Omega$,
$$
\Omega e^{iX'_{A,B}}P'_{A}e^{-iX'_{A,B}}\Omega^*=e^{i\Omega X'_{A,B}\Omega^*}\Omega P'_{A}\Omega^*e^{-i\Omega X'_{A,B}\Omega^*}=\Omega P'_{B}\Omega^*.
$$
Recall from Remark \ref{propiedad omega}.3  that,  since $A$ and $B$ are selfadjoint,  $\Omega P_{A}\Omega^*=I-P_{A}$ and $\Omega P_{B}\Omega^*=I-P_{B}$. Since $\Omega$ reduces to the generic part, we have also
$\Omega P'_{A}\Omega^*=I-P'_{A}$ and $\Omega P'_{B}\Omega^*=I-P'_{B}$. Then
$$
e^{i\Omega X'_{A,B}\Omega^*}(I-P'_{A})e^{-i\Omega X'_{A,B}\Omega^*}= I-P'_{B},
$$
which  implies that $e^{i\Omega X'_{A,B}\Omega^*}P'_{A}e^{-i\Omega X'_{A,B}\Omega^*}=P'_{B}$. Clearly $\|\Omega X'_{A,B}\Omega^*\|=\|X'_{A,B}\|\le\pi/2$ and $\Omega X'_{A,B}\Omega^*$ is selfadjoint. Finally, an operator $X$ is $P$-codiagonal if and only if (for any unitary operator $U$) $UXU^*$ is $UPU^*$-codiagonal. This implies that $\Omega X'_{A,B} \Omega^*$ is $\Omega P'_{A}\Omega^*=(I-P'_{A})$-codiagonal. But clearly, being $(I-P)$-codiagonal is the same as being $P$-codiagonal, thus $\Omega X'_{A,B} \Omega^*$ is $P'_{A}$-codiagonal. Therefore, by the uniqueness of the geodesic in the generic part, it follows that
$$
\Omega X'_{A,B} \Omega^*=X'_{A,B}.
$$
\end{proof}

Note that operators commuting with $\Omega$ are of the form
$$
X=\left(\begin{array}{cc} Y & Z \\ -Z & Y \end{array}\right).
$$

\begin{lem}\label{lema 44}
Let $A:D(A)\subset\h\to\h$ be a selfadjoint operator and $U=\left(\begin{array}{cc} V & W \\ -W & V \end{array}\right)$  a unitary operator in $\h\times\h$  (which commutes with $\Omega$). Then
$$
U(Gr(A)) \hbox{ is a graph } \iff N(V^*+AW^*)=\{0\}.
$$
\end{lem}
\begin{proof}
Clearly $U(Gr(A))$ is a graph if and only if it does not contain vectors $(0,h)$ with $h\ne 0$. That is, for any $h\ne 0$, the solution of the  equation $U(f,g)=(0,h)$ does not belong to $Gr(A)$. This happens if (for any $h\ne 0$), $g\ne Af$, i.e.,
$-AW^*h\ne V^*h$, or $(V^*+AW^*)h\ne 0$.
\end{proof}
In the case when $A$ is arbitrary and $B=0$ , for $U=e^{i X_{0,A}}$, the condition of the Lemma reads $\{0\}=N(V^*)=N((I+A^2)^{-1/2})$, which clearly holds.

Finally we address the second question in \ref{cuestiones}. The following example shows that even in the simplest situation, the geodesic may wander outside $\Gamma$.

\begin{ejem}\label{ejemplos}

Let us consider now the simplest possible setting: $\h=\mathbb{C}$ and $A,B:\mathbb{C}\to\mathbb{C}$ given by $A(z)=az$, $B(z)=bz$, for $a,b\in\mathbb{R}$. Let us examine the geodesic between the complex lines $Gr(A)=\{(z,az): z\in\mathbb{C}\}$ and $Gr(B)=\{(z,bz): z\in\mathbb{C}\}$. Clearly $Gr(A)\cap Gr(B)=Gr(A)^\perp\cap Gr(B)^\perp=\{0\}$ if $a\ne b$,  and $Gr(A)\cap Gr(B)^\perp=Gr(A)^\perp\cap Gr(B)=\{0\}$ if $1+ab\ne 0$. Assume these conditions (i.e., that $Gr(A)$ and $Gr(B)$ are in generic position). In order to keep the example  more plain, assume also that $-1<a<0$ and $b=-a$. The projections $P_a$ and $P_b$ in $\mathbb{C}^{2\times 2}$ are given by
$$
P_a=\left( \begin{array}{cc} \frac{1}{1+a^2} & \frac{a}{1+a^2} \\ \frac{a}{1+a^2} & \frac{a^2}{1+a^2} \end{array} \right) \ \hbox{ and } \ P_b=\left( \begin{array}{cc} \frac{1}{1+a^2} & \frac{-a}{1+a^2} \\ \frac{-a}{1+a^2} & \frac{a^2}{1+a^2} \end{array} \right).
$$
The unitary operator that diagonalizes $P_a$ is $U_a=\left( \begin{array}{cc} \frac{1}{(1+a^2)^{1/2}} & \frac{-a}{(1+a^2)^{1/2}} \\ \frac{a}{(1+a^2)^{1/2}} & \frac{1}{(1+a^2)^{1/2}} \end{array} \right)$: 
$$
U_a\left(\begin{array}{cc} 1 & 0 \\ 0 & 0 \end{array}\right) U_a^*=P_a.
$$
Note that $U_a$ is a rotation of angle $\theta=\arctan(a)$. 

To compute the exponent of the geodesic $\delta$ that joins $\delta(0)=P_a$ with $\delta(1)=P_b$, we must identify the angle $0<x\le\pi/2$ such that for $c=\cos(x)$, $s=\sin(x)$, we have 
$$
U_a\left(\begin{array}{cc} c^2 & cs \\ cs & s^2 \end{array} \right)U_a^*=P_b, \ \hbox{ i.e., } \ \left(\begin{array}{cc} c^2 & cs \\ cs & s^2 \end{array} \right)=U_a^*P_b U_a.
$$

After elementary computations this gives $c=\frac{1-a^2}{1+a^2}$ and $s=\frac{-2a}{1+a^2}$ which satisfy that $x=\arccos(\frac{1-a^2}{1+a^2})<\pi/2$.
 The exponent of $X_{A,B}$ of the geodesic $\delta$, is obtained from the matrix  
$$
X'=\left(\begin{array}{cc} 0 & -ix \\ ix & 0 \end{array}\right)
$$
in the system of coordinates that diagonalizes $P_b$.  Then  $e^{itX'}=\left(\begin{array}{cc} \cos(tx) & \sin(tx) \\ -\sin(tx) & \cos(tx) \end{array}\right)$. The actual exponent $X_{A,B}$ is $U_aX'U_a^*$, and the unitary $e^{itX_{A,B}}$ that pushes the geodesic $\delta$  is therefore 
$$
e^{itX_{A,B}}=U_ae^{itX'}U_a^*=\left(\begin{array}{cc} \cos(tx) & \sin(tx) \\ -\sin(tx) & \cos(tx) \end{array}\right),
$$
because $U_a$ and $e^{itX'}$ are both rotations.
Therefore, if we apply Lemma \ref{lema 44} in this case, $V=V^*=\cos(tx)$, $W=W^*=\sin(tx)$ and $A=a$, we need to check if the number $\cos(tx)+a\sin(tx)=0$ for some $t\in(0,1)$. This is equivalent to $\tan(tx)=-\frac{1}{a}$ for some $-1<a<0$, where $x=\arccos(\frac{1-a^2}{1+a^2})$. Or $\arctan(-\frac1a)=t\arccos(\frac{1-a^2}{1+a^2})$ for some $t\in(0,1)$, $-1<a<0$. That is, we need to check if there exists  an $a\in(-1,0)$ with $\arctan(-\frac1a)<\arccos(\frac{1-a^2}{1+a^2})$, or equivalently, with 
$$
-\frac1a<\tan\left(\arccos(\frac{1-a^2}{1+a^2})\right)=\frac{\sqrt{1-\left(\frac{1-a^2}{1+a^2}\right)^2}}{\frac{1-a^2}{1+a^2}} 
=-\frac{2a}{1-a^2},
$$
which happens choosing $-1<a<-1/\sqrt{3})$.

\end{ejem}

\section{Graphs and the restricted Grassmannian}

Le us recall the definition of   the so-called {\it restricted Grassmannian} \cite{sato}, \cite{segalwilson}:
\begin{defi}
Let $P$ be an orthogonal projection in $\j=\j_+\oplus\j_-$. Denote $P_+=P_{\j_+}$ and $P_-=P_{\j_-}$. Then $P$ belongs to the restricted Grassmannian $\p_{res}(\j_+)$  (induced by this decomposition) if and only if
\begin{itemize}
\item
$P_+P\big|_{R(P)}:R(P)\to \j_+ $ is a Fredholm operator;
\item
$P_-P\big|_{R(P)}:R(P)\to\j_-$ is  compact operator.
\end{itemize}
The index of $P$ is the index of the above Fredholm operator, and is given by
$$
{\rm ind}_{\j_+}(P)=\dim R(P)\cap\j_--\dim N(P)\cap\j_+.
$$
\end{defi}

The connected components of $\p_{res}(\j_+)$ are parametrized by this index. The connected components of $\p_{res}(\j_+)$ are $\p_{res}^m(\j_+)$ given by
$$
\p_{res}^m(\j_+)=\{P\in\p_{res}(\j_+): {\rm ind}_{\j_+}(P)=m\},
$$
for $m\in\mathbb{Z}$.
\begin{prop}
If $T\in\k(\h)$, then $P_{Gr(T)}\in\p_{res}^0(\h\times\{0\})$ for the decomposition given by $\h\times\h=\h\times\{0\}\oplus \{0\}\times\h$. 

Conversely, if $T$ is closed and $P_{Gr(T)}\in \p_{res}(\j_+)$, then $T$ is compact.
\end{prop}
\begin{proof}
The operator ${\bf F}:=\Pi_1P_{Gr(T)}:Gr(T)\to\h\times\{0\}$ is given by ${\bf F}(f,Tf)=(f,0)$. Clearly ${\bf F}$ is onto and $N({\bf F})=\{0\}$, i.e., is an isomorphism, thus  a Fredholm operator of index zero.

${\bf K}:=\Pi_2P_{Gr(T)}:Gr(T)\to \{0\}\times\h$ is given by ${\bf K}(f,Tf)=Tf$ is  a compact operator.

Suppose now that $P_{Gr(T)}\in\p_{res}(\j_+)$ for $T$ a closed operator. Then the facts that  $(I+|T|^2)^{-1}$ is Fredholm, and $T(I+|T|^2)^{-1}$ is compact, imply that $T$ is compact: indeed, consider the projection $\pi:\b(\h)\to\b(\h)/\k(\h)$ onto the Calkin algebra,  we have that $\pi((I+|T|^2)^{-1})$ is invertible and that $\pi(T)\pi((I+|T|^2)^{-1})=0$, thus $\pi(T)=0$, i.e., $T$ is compact.
\end{proof}

The following group acts in the restricted Grassmannian of $\j=\j_+\oplus\j_-$, for any decomposition, with $\dim\j_+=\dim\j_-$ ($=\infty$):
\begin{equation}\label{grupo Uinfinito} 
\u_\infty(\j)=\{U\in\b(\j): U-I\in\k(\j)\}.
\end{equation}
This group is well known, it is  one of the so called classical infinite dimensional  Banach-Lie groups \cite{harpe}. It is not difficult to prove that
$$
\u_\infty(\j)=\exp\{iX: X^*=X\in\k(\j), \|X\|\le \pi\}.
$$
It was shown in \cite{beltitaratiutumpach} (see also \cite{rectifiable}), that $\u_\infty(\j)$ acts transitively in each connected component of $\p_{res}(\j_+)$. Moreover, in \cite{rectifiable} it was shown that (adapted to our situation):
\begin{prop}{\rm (}Theorem 6.1 in {\rm \cite{rectifiable})}
Let ${\bf P}, {\bf Q}\in\p_{res}(\h\times\{0\})$ in the same connected component (i.e., with the same index with respect to the decomposition $\h\times\h=\h\times\{0\} \ \oplus \ \{0\}\times \h$). Then there exists a compact operator ${\bf X}^*={\bf X}$, $\|{\bf X}\|\le \pi/2$, which is codiagonal with respect to ${\bf P}$ and ${\bf Q}$, such that 
${\bf Q}=e^{i{\bf X}}{\bf P}e^{-i{\bf X}}$.  
\end{prop}
In other words, ${\bf P}$ and ${\bf Q}$ in the same component of $\p_{res}(\h\times\{0\})$ are joined by a minimal geodesic, with compact velocity.
\begin{coro}
Let $S,T\in\k(\h)$. Then there exists a minimal geodesic  joining the graphs of $S$ and $T$, which lies inside  $\p_{res}^0(\h\times\{0\})$.
\end{coro}

\section{Unitary equivalent operators}

Let $T:D(T)\subset\h\to\h$ be closed and densely defined, and $U$ a unitary operator in $\h$. Define $UTU^*$ with domain $D(UTU^*)=UD(T)$,  by $UTU^*f=UTg$, if $g\in D(T)$, $g=U^*f$. Clearly $UTU^*$ is a closed operator with this (dense) domain. If $B=UTU^*$, we say that $T$ and $B$ are unitarily equivalent. Denote by
$$
{\bf U}:\h\times\h\to\h\times\h, \ {\bf U}(f,g)=(Uf,Ug).
$$
Apparently,
$$
Gr(UTU^*)=\{(f,UTU^*f): f\in UD(T)\}=\{(Ug, UTg): g (=U^*f)\in D(T)\}
$$
$$
={\bf U}(Gr(T)).
$$
It follows  that
\begin{prop}
If $T:D(T)\subset\h\to\h$ and $B:D(B)\subset\h\to\h$ are closed, densely defined  and unitarily equivalent, then $Gr(T)$ and $Gr(B)$ can be joined in $\p(\h\times\h)$ by a smooth curve of graphs, of unitarily equivalent closed and densely defined operators.
\end{prop}
\begin{proof}
Let $U$ be  a unitary operator such that $B=UTU^*$ (with $UD(T)=D(B)$). There exists a bounded selfadjoint operator $X$, $\|X\|\le\pi$, such that $U=e^{iX}$. Then $U(t)=e^{itX}$ is a smooth curve of unitaries, and $T(t)=U(t)TU^*(t)$ is a path of unitarily equivalent closed and densely defined operators, with
$$
P_{T(t)}={\bf U}(t)P_T{\bf U}^*(t),
$$
for ${\bf U}(t)(f,g)=(U(t)f,U(t)g)$ as above. $P_{T(t)}$  is a smooth path in $\p(\h\times\h)$ joining $P_T$ and $P_B$, and consists of graphs of closed (densely defined) operators. 
\end{proof}

\begin{ejem}
In this example we consider two classical operators, which are unitary equivalent and joined by a unique geodesic. 
Consider $\h=L^2(\mathbb{R})$ with Lebesgue measure, $T:D(T)\subset L^2(\mathbb{R})\to L^2(\mathbb{R})$, given by $D(T)=\{f\in L^2(\mathbb{R}): tf(t)\in L^2(\mathbb{R})\}$ and $Tf(t)=2\pi tf(t)$. Consider also $B:D(B)\subset L^2(\mathbb{R})\to L^2(\mathbb{R})$, where $D(B)=\{g\in L^2(\mathbb{R}): g' \hbox{ exists } pp \hbox{ and } g'\in L^2(\mathbb{R})\}$, the Sobolev space $H^1_2(\mathbb{R})$,  and $Bg=-ig'$. Clearly $T$ and $B$ are selfadjoint operators, that are unitarily equivalent. The Fourier-Plancherel transform $\f:L^2(\mathbb{R})\to L^2(\mathbb{R})$, $\f f(x)=\int_{-\infty}^\infty f(t)e^{-itx} dt$ intertwines $T$ and $B$: $\f T=B\f $, so that $B=\f T\f^*$, and also $T\f=-\f B$.  In particular, $Gr(T)$ and $Gr(B)$ are connected in $\p(L^2(\mathbb{R})\times L^2(\mathbb{R}))$ by a smooth curve of graphs. The intermediate operators $A(t)=e^{itX}Ae^{-itX}$ are clearly unbounded. Here $X=\log \f$. Since $\f^4=I$, $\f$ has four eigenvalues $1, -1, i, -i$, with (projections onto the) eigenspaces $E_1, E_{-1}, E_i, E_{-i}$, so that $X=\log \f=-\pi E_{-1}+\frac{\pi}{2}E_i-\frac{\pi}{2}E_{-i}$. These projections $E_j$ are polynomial expressions in terms of $\f$ ($E_1$ is not needed in the expression of $X$, it corresponds to the eigenvalue $0$ of $X$):
$$
E_{-1}=\frac14(I-\f+\f^2-\f^3), E_i=\frac14(I-i\f-\f^2+i\f^3), E_{-i}=\frac14(I+i\f-\f^2-i\f^3),
$$
so that
$$
X=\frac{\pi}{4}\left(-I+(1-i)\f-\f^2+(1+i)\f^3\right).
$$
The curve 
$$
P_{T(t)}=\left(\begin{array}{cc} e^{itX} & 0 \\ 0 & e^{itX}\end{array}\right)P_A\left(\begin{array}{cc} e^{-itX} & 0 \\ 0 & e^{-itX}\end{array}\right)
$$
of graphs, joining $Gr(T)$ and $Gr(B)$, induced by the operator $X=\log\f$ is not a geodesic of $\p(\h\times\h)$. Indeed,  the exponent ${\bf X}=\left(\begin{array}{cc} X & 0 \\ 0 & X \end{array}\right)$ is not codiagonal with respect to $P_T$: it does not hold that ${\bf X}(Gr(T))$ is orthogonal to $Gr(T)$. To see this, pick $\psi_0(x)=e^{-\frac{x^2}{2}}$, with $T\psi_0(x)=-xe^{-\frac{x^2}{2}}=-\psi_1(x)$, which are eigenvectors of $\f$, with eigenvalues $1$ and $i$, respectively. Then $T\psi_0=ixe^{-\frac{x^2}{2}}=i\psi_1(x)$ and
$$
{\bf X}(\psi_0, T\psi_0)=(X\psi_0,-iX\psi'_0(x))=(0, -i\frac{\pi}{2}\psi_1),
$$
so that 
$$
\langle {\bf X}(\psi_0, T\psi_0), (\psi_0, T\psi_0)\rangle=\langle (0, -i\frac{\pi}{2}\psi_1),(\psi_0,i\psi_1)\rangle=-\frac{\pi}{2}\|\psi_1\|_2^2\ne 0,
$$
and therefore ${\bf X}$ is not $P_T$-codiagonal.

Note also that $Gr(T)$ and $Gr(B)$ are in generic position. 
Indeed, first note that $Gr(T)\cap Gr(B)=\{0\}$: pick $f\in Gr(T)\cap Gr(B)$, then $f\in D(T)\cap D(B)$, i.e., $tf(t)\in L^2(\mathbb{R})$, $f'$ exists a.e., $f'\in L^2(\mathbb{R})$ and $2\pi tf(t)=-if'(t)$. This clearly implies that $f$ is absolutely continuous, and thus $f$ and $t f$ are continuous, and therefore $g$ is $C^1$ in $\mathbb{R}$. Then   $f$ is one of the  smooth  solutions of the equation $2\pi i xy=\frac{dy}{dx}$, which  are of the form $y=C e^{2\pi i x^2}$. These  satisfy $|y|=C$, and since $f\in L^2(\mathbb{R})$, it must be $f=0$. Then also $Gr(T)^\perp\cap Gr(B)^\perp=\Omega(Gr(T)\cap Gr(B))=\{0\}$. Similarly, if $f\in Gr(T)\cap Gr(B)^\perp$, then $f=i g'$ and $2\pi tf(t)=g$. Then, as above, $g$ is continuous, thus $t f$ is continuous. This implies that $f$ is continuous and $g$ is $C^1$  in any closed interval ${\bf I}$,  such that $0\notin {\bf I}$. Then  $f$ is a solution of $2\pi ix\frac{dy}{dx}=(1-2\pi i)y$ in the interval ${\bf I}$. These solutions are of the form $y=C\frac{1}{x}e^{\frac{i}{2\pi} \ ln(|x|)}$, and have modulus $|y|=C \frac{1}{|x|}$ for $x\in {\bf I}$. For an integer $n\ge 1$, put ${\bf I}=[\frac1n,n]$. Then there exists a constant $C_n\ge 0$ such that   $|f(t)|=\frac{C_n}{|t|}$ in $[\frac1n,n]$. Then 
$$
+\infty>\int_{-\infty}^\infty |f(t)|^2dt\ge\int_{\frac1n}^n|f(t)|^2dt=C_n^2 \int_{\frac1n}^n \frac{1}{|t|^2}dt=C_n^2(n-\frac1n),
$$
i.e., $C_n\to 0$ if $n\to \infty$. This implies that $f=0$ in $(0,+\infty)$. Similarly, reasoning analogously in the negative axis, it must be $f=0$. Then, also  $Gr(T)^\perp\cap Gr(B)=\Omega(Gr(T)\cap Gr(B)^\perp)=\{0\}$. 

It follows that there exists a unique geodesic (up to reparametrization) joining $P_T$ and $P_B$.
\end{ejem}

\section{Common complements}

If $S$ and $T$ are bounded  operators, then $Gr(S)$ and $Gr(T)$ have a common complement: clearly $Gr(S)\dot{+}\left(\{0\}\times\h\right)=Gr(T)\dot{+}\left(\{0\}\times\h\right)=\h\times\h$. This subspace $\{0\}\times\h$ is no longer a complement for $Gr(S)$ if $S$ is unbounded, as remarked above, $Gr(S)\oplus\{0\}\times\h=D(S)\times\h$. Thus it is a fair question if there exists a common complement for $Gr(S)$ and $Gr(T)$ if $S$ or $T$ are unbounded. 

M. Lauzon and S. Treil in \cite{lauzontreil} characterized pairs of closed subspaces $\s,\t$ of a Hilbert space $\j$ such that there exists a closed subspace $\z\subset\j$ with $\s\dot{+}\z=\j=\t\dot{+}\z=\j$. Shortly after, J. Giol \cite{giol} proved  the following result:
\begin{teo}\label{teo giol}{\rm (Proposition 6.2 in  \cite{giol})}
The closed subspaces $\s,\t\subset\j$ have a common complement if and only if there exists an orthogonal projection $P$ in $\j$ such that $\|P_\s-P\|<1$ and $\|P_\t-P\|<1$.
\end{teo}

\begin{rem}\label{no densidad}
This result is meaningful for the topology of $\p(\j)$. It means that for any $P\in\p_\infty(\j)$, although it holds that $\|P-Q\|\le 1$ for any other $Q\in\p_\infty$, the ball $\mathbb{B}_1(P)=\{E\in\p_\infty(\j): \|E-P\|<1\}$ is not dense in $\p_\infty(\j)$: if $P$ and $Q$ do not have a common complement, then 
$\mathbb{B}_1(P)\cap\mathbb{B}_1(Q)=\emptyset$.
\end{rem}

\begin{rem}
In \cite{complemento comun}, it was remarked that if $\s$ and $\t$ are  closed subspaces, such that there exists a  geodesic $\delta$ joining them, then $\s$ and $\t$ have a common complement.
\end{rem}

\begin{prop}
If $S$ is bounded and $T$ is closed and densely defined, then $Gr(S)$ and $Gr(T)$ have a common complement.
\end{prop}
\begin{proof}
By Theorem \ref{42}, we know that $\|P_S-\Pi_1\|<1$.  Let $\delta(t)$ be the unique geodesic (Theorem 5.7 in \cite{eder}) between $\delta(0)=\Pi_1$ and $\delta(1)=P_T$: $\delta(t)=e^{itX}\Pi_1 e^{-itX}$, with $X^*=X$, codiagonal with respect to $\Pi_1$ such that $\|X\|\le \pi/2$.
As remarked in the section of preliminaries, the fact   $X$ is $\Pi_1$-codiagonal means that $X$ anti-commutes with the symmetry $2\Pi_1-I$. Then, for any $t_0\in(0,1)$
$$
P_T-\delta(t_0)=e^{iX}\Pi_1e^{-iX}-e^{i t_0 X}\Pi_1e^{-i t_0 X}=\frac12\{e^{iX}(2\Pi_1-I)e^{-iX}-e^{it_0X}(2\Pi_1-I)e^{-it_0X}\}
$$
$$
=\frac12\{e^{2iX}(2\Pi_1-I)-e^{2it_0X}(2\Pi_1-I)\}=\frac12 e^{2it_0X}\left(e^{2i(1-t_0)X}-I\right)(2\Pi_1-I).
$$
Then, 
$$
\|P_T-\delta(t_0)\|=\frac12 \|e^{i(1-t_0)X}-I\|=\sup\{|e^{i2(1-t_0)s}-1|: s\in\sigma(X)\}
$$
$$
\le\sup\{|1-e^{ir}|: t\in[(1-t_0)\pi,(1-t_0)\pi]\},
$$
because $\|X\|\le\pi/2$. Clearly, since $1-t_0<1$,  $|1-e^{ir}|<2$ for such $r$. Then $\|P_T-\delta(t_0)\|<1$  for any $0<t_0<1$. 

On the other hand, by a continuity argument, the fact that $\|P_S-\delta(0)\|<1$ implies that for sufficiently small $t_0$, one has $\|P_S-\delta(t_0)\|<1$. That is, we have found an intermediate projection $\delta(t_0)$, at distance less than $1$ from $P_\s$ and $P_\t$.
\end{proof}

\begin{lem}\label{lema 74}
Let $T$ be a closed operator with dense domain and dense range, such that there exists $r>0$ with 
$$
\|Tf\|\ge r\|f\|, \hbox{ for all } f\in D(T).
$$
Then
$\|P_T-\Pi_2\|<1$.
\end{lem}
\begin{proof}
The fact that $R(T)$ is dense means that in the  polar decomposition $T=V|T|$, $V$ is a unitary operator ($N(T^*)=R(T)^\perp=\{0\}$), i.e., $|T|$ has dense range and trivial nullspace.

Recall ${\bf V}=\left(\begin{array}{cc} I & 0  \\ 0 & V \end{array}\right)$ from Remark \ref{remark 21}. Note that in this case ${\bf V}$ is a unitary operator. Then
$$
{\bf V}^*(P_T-\Pi_2){\bf V}=P_{|T|}-\Pi_2.
$$
Therefore we may compute the norm of the right hand difference. Note that $|T|^2(I+|T|^2)^{-1}-I=-(I+|T|^2)^{-1}$. Then 
$$
P_{|T|}-\Pi_2=\left(\begin{array}{cc} (I+|T|^2)^{-1} & (I+|T|^2)^{-1}|T| \\ |T|(I+|T|^2)^{-1} & -(I+|T|^2)^{-1}\end{array}\right).
$$

We can use the Spectral Theorem for $|T|$ (in the multiplication operator version given in Section 1), and consider $\h=L^2(\mu)$ for $\mu$ finite and $|T|=M_\varphi$, where $\varphi\ge 0 \ \mu$-a.e.  The condition 
$$
\|Tf\|=\||T|f\|\ge r \|f\|
$$
clearly implies that $\varphi(t)\ge r \ \mu$-a.e. We have to estimate the (operator) norm of the matrix
$$
\left(\begin{array}{cc} M_{\frac{1}{1+\varphi^2}} & M_{\frac{\varphi}{1+\varphi^2}} \\ M_{\frac{\varphi}{1+\varphi^2}} & M_{\frac{-1}{1+\varphi^2}} \end{array}\right)
$$
acting in $L^2(\mu)\times L^2(\mu)$. This is the norm in the C$^*$-algebra $M_2(L^\infty(\mu))$  of the matrix function
$$
A_\varphi=\left( \begin{array}{cc} \frac{1}{1+\varphi^2} & \frac{\varphi}{1+\varphi^2} \\ \frac{\varphi}{1+\varphi^2} & \frac{-1}{1+\varphi^2} \end{array}\right).
$$
Indeed, the map 
$$
M_2(L^\infty(\mu))\ni  \left(\begin{array}{cc} \varphi_{11} & \varphi_{12} \\ \varphi_{21} & \varphi_{22} \end{array} \right) \mapsto \left(\begin{array}{cc} M_{\varphi_{11}} & M_{\varphi_{12}} \\ M_{\varphi_{21}} & M_{\varphi_{22}} \end{array} \right)\in\b(L^2(\mu)\times L^2(\mu))
$$
is a faithful $*$-representation.  Note also that 
$$
A_\varphi^2=\left(\begin{array}{cc} \frac{1}{(1+\varphi^2)^2} & 0 \\ 0 & \frac{1}{(1+\varphi^2)^2} \end{array} \right), \ \hbox{ so that } \ \|A_\varphi\|\le \|\frac{1}{1+\varphi^2}\|_{L^\infty(\mu)}\le \frac{1}{1+r^2}<1.
$$
Therefore  $\|P_{|T|}-\Pi_2\|<1$.
\end{proof}
If $N(T)\ne\{0\}$, this norm inequality does not hold: if $0\ne f\in N(T)$, then $(f,0)\in Gr(T)$ and then $(P_T-\Pi_2)(f,0)=(f,0)$, and then $\|P_T-\Pi_2\|=1$. Similarly, if there exist $f_n\in\h$ such that $\|f_n\|=1$ and $Tf_n\to 0$, then $\|P_T-\Pi_2\|=1$.

\begin{coro}
Let $S,T$ be operators as in the statement of Lemma \ref{lema 74}. Then $Gr(T)$ and $Gr(S)$ have a common complement. 
\end{coro}
\begin{proof}
By  Lemma \ref{lema 74}, $\|P_T-\Pi_2\|<1$ and $\|P_S-\Pi_2\|<1$. Therefore, by the result of Giol \cite{giol} transcribed here in Theorem \ref{teo giol}, $Gr(T)$ and $Gr(S)$ have a common complement.
\end{proof}

\begin{ejem}
Let $A$ and $B$ be diagonal operators, with respect to the orthonormal bases $\{e_n: n\ge 1\}$ and $\{f_k: k\ge 1\}$ of $\h$, and eigenvalues $\alpha_n$ and $\beta_k$ such that $|\alpha_n|, |\beta_k|\to+\infty$. Let $G$ and $L$ be invertible operators. Then $T=GAG^{-1}$ (with dense domain $G(D(A))$) and $S=LBL^{-1}$ (with dense domain $L(D(B))$) satisfy the hypothesis of Lemma \ref{lema 74}. Therefore $Gr(S)$ and $Gr(T)$ have a common complement.
\end{ejem}
\bigskip
The case that remains outside these affirmative situations, is that of two unbounded closed operators, one of them nonselfadjoint. The following example shows that the answer is, in general,  negative.
\begin{ejem}
Consider $\h=\ell^2(\mathbb{N})$, $A$ the diagonal operator with $1,2,3,\dots$ in the diagonal, and $T={\bf S}A$, with ${\bf S}$ the unilateral shift. Clearly $T^*T=A^2$ so that $|T|=A$. The domain of $A$ and $T$ is $D=\{f\in\ell^2(\mathbb{N}): (nf_n)\in\ell^2(\mathbb{N})\}$. Note the following elementary facts:

1) $Gr(A)^\perp\cap Gr(T)=\{0\}$: if $f=Tf$ and $-Af=g$, for $f=(f_1,f_2,\dots)$, $g=(g_1,g_2,\dots)$ in $\ell^2(\mathbb{N})$, then 
	$$
	(f_1,f_2,f_2,f_3,\dots)=(0,g_1,2g_2,3g_3,\dots) \ \hbox{ and } \ (g_1,g_2,g_3,\dots)=-(f_1,2f_2,3f_3,\dots)
	$$
	clearly imply that $f=g=0$.

2) $Gr(A)\cap Gr(T)^\perp$ is one dimensional: $-T^*f=A{\bf S}^*f=g$ and $f=Ag$ mean 
	$$
	(f_1,f_2,f_3,\dots)=(g_1,g_2,g_3,\dots) \ \hbox{ and } \ -(f_2,2f_3,3f_4,\dots)=(g_1,g_2,g_3,\dots),
	$$
	which after elementary computations lead to $f_n=-n^2f_{n+1}$ and $g_n=-n(n+1)g_{n+1}$. These mean that the sequence $f$ (in the intersection) is a multiple of the sequence $\kappa$ given by $\kappa_n=\frac{(-1)^{n+1}}{\left((n-1)!\right)^2}$, which clearly belongs to the domain $D$.

3) The product $P_A^\perp P_T$ is compact. Indeed, note that 
	$$
	P_A^\perp P_T=\left(I-\left(\begin{array}{cc} (I+A^2)^{-1} & (I+A^2)^{-1}A \\ A(I+A^2)^{-1} & A^2(I+A^2)^{-1}\end{array}\right)\right) \left(\begin{array}{cc} (I+A^2)^{-1} & (I+A^2)^{-1}T^* \\ T(I+A^2)^{-1} & T(I+A^2)^{-1}T^*\end{array}\right)
	$$
	$$
	=\left(\begin{array}{cc} A^2(I+A^2)^{-1} & -(I+A^2)^{-1}A \\ -A(I+A^2)^{-1} & (I+A^2)^{-1}\end{array}\right)\left(\begin{array}{cc} (I+A^2)^{-1} & (I+A^2)^{-1}T^* \\ T(I+A^2)^{-1} & T(I+A^2)^{-1}T^*\end{array}\right)
	$$
	whose entries are compact. Indeed:
\begin{itemize}
\item
The $1,1$ entry is $A^2(I+A^2)^{-2}- A(I+A^2)^{-1}T(I+A^2)^{-1}$. The operator $A^2(I+A^2)^{-2}$ is compact: it  is diagonal, with entries  $\frac{n^2}{(1+n^2)^2}$, which tend to zero. Its square root $A(I+A^2)^{-1}$ is also compact, the operator $T(I+A^2)^{-1}=T(I+|T|^2)^{-1}$ is bounded; then $A(I+A^2)^{-1}T(I+A^2)^{-1}$ is compact.
\item
The $1,2$ entry is $A^2(I+A^2)^{-2}T^*+A(I+A^2)^{-1}T(I+A^2)^{-1}T^*$. The operator $B=A^2(I+A^2)^{-2}T^*$ is compact: $BB^*=A^6(I+A^2)^{-4}$ is diagonal with eigenvalues $\frac{n^6}{(1+n^4)^2}\to 0 $. The operator $A(I+A^2)^{-1}T(I+A^2)^{-1}T^*$ is the product of  $A(I+A^2)^{-1}$, which is compact, times  $T(I+A^2)^{-1}T^*=T(I+|T|^2)^{-1}T^*$, which is bounded.
\item
The $2,1$ entry is $-A(I+A^2)^{-2}+(I+A^2)^{-1}T(I+A^2)^{-1}$. The first summand is clearly compact; the second summand is $(I+A^2)^{-1}$ compact, times $T(I+A^2)^{-1}$ bounded.
\item

The $2,2$ entry is $A(I+A^2)^{-2}T^*+(I+A^2)^{-1}T(I+A^2)^{-1}T^*$. The first summand is $A(I+A^2)^{-1}$ compact, times $(I+A^2)^{-1}T^*$ bounded. The second summand $(I+A^2)^{-1}$ compact, times $T(I+A^2)^{-1}T^*$ bounded.
\end{itemize}

Then we can apply a consequence of the results of Lauzon and Treil \cite{lauzontreil}, in the comments in Section 4 in \cite{complemento comun}, which say that two subspaces $\s$ and $\t$, such that:
\begin{itemize}
\item
$\dim \s\cap\t^\perp \ne\dim\s^\perp\cap\t$, and
\item
either $P_\s^\perp P_\t$ or $P_\s P_\t^\perp$ is compact,
\end{itemize}
do not have  a common complement. 
\end{ejem}

{\bf Statements and declarations}

Data sharing not applicable to this article as no data sets were generated or analyzed during the current study.

No financial or non-financial interests  are directly or indirectly related to the work submitted for publication.



\end{document}